\documentclass{article}
\usepackage{inputenc}
\usepackage{amscd}
\usepackage{amssymb}
\usepackage{pstricks}
\newcommand{\R}{\mathbb{R}}
\newcommand{\N}{\mathbb{N}}

\newtheorem{prop}{Proposition}
\newtheorem{rem}{Remark}
\newtheorem{thm}{Theorem}
\newtheorem{exa}{Example}
\newenvironment{proof}{{\noindent\it{Proof}}.}{\hfill$\Box$}

\title{Persistence and periodic solutions 
in systems of delay differential equations}

\author{
{Pablo Amster}
\thanks{Partially supported by UBACyT Project 20020160100002BA. E-mail: \texttt{pamster@dm.uba.ar}}\\%
Departamento de Matem\'atica - FCEyN\\ 
Universidad de Buenos Aires \& IMAS-CONICET.
\\
{Melanie Bondorevsky}%
\thanks{Partially supported by UBACyT Project 20020160100002BA. E-mail: \texttt{mbondo@dm.uba.ar}}\\%
Departamento de Matem\'atica - FCEyN\\ Universidad de Buenos Aires \& IMAS-CONICET\\ 
}
\date{}

\begin{document}

\maketitle

\begin{abstract}

We study semi-dynamical systems associated to delay differential equations. We give a simple criteria to obtain weak and strong persistence and provide sufficient conditions to guarantee uniform persistence.
Moreover, we show the existence of non-trivial
$T$-periodic solutions via
topological degree techniques.
Finally, we prove that, in some sense, the conditions are also necessary. 

\end{abstract}

\section*{Introduction}

Let us consider the system of non-autonomous differential equations:
$$x'(t) = f(t,x(t)),$$
where $x(t)=(x_1(t), x_2(t),..., x_N(t))\in\R^N$, $f:\R\times\overline\Omega\to\R^N$ is a continuous function 
and $\Omega\subset\R^N$ is a bounded domain with smooth boundary.
It is well known that if the solutions are uniquely defined over $\R$ and 
then the flow associated to the system:
$$
\Phi:\R\times\overline\Omega\to\overline\Omega,
$$
given by
$$
\Phi(t,x_0) =x(t)
$$
induces a dynamical system.

\bigskip

In this work, we consider semi-dynamical systems associated to delay
differential equations which, in turn, describe the behaviour of biological processes. In particular, we shall restrict our attention to study solutions of systems which must remain non-negative for all times. We are interested in establishing conditions in order to guarantee that non-trivial solutions with non-negative initial condition remain in the interior of the positive cone of $\R^N$  as time evolves, i.e. $x_j(t)>0$ for all $j$ and all $t>0$.
More precisely, we are
interested in studying the notion of \emph{persistence} from the following point of view: we shall say that a system is \emph{persistent} if it guarantees the survival of the total population. 
It is clear that this notion of persistence does not imply the survival of all the species.

Various definitions of persistence have been given over the literature (see e.g. \cite{BFW}, \cite{ST}): a system is said to be \emph{weakly persistent} if solutions do not asymptotically approach the origin as  $t\to+\infty$; \emph{strong persistence} means that each solution is eventually bounded away from zero; \emph{uniform persistence} means that solutions are eventually uniformly bounded away from zero.  We shall state appropriate definitions below. 

The first results in this subject belongs to 
Butler, Freedman and Waltman \cite{BFW}  that proved that if $X$ is a locally compact metric space and the flow verifies some specific conditions, uniform persistence is equivalent to weaker forms of persistence. Namely, uniform persistence is equivalent to weak persistence.

With similar assumptions, Butler and Waltman \cite{BW} showed a necessary and sufficient condition for a dynamical system to be uniformly persistent.
A theorem which establishes a necessary and sufficient condition for a semi-dynamical system to be uniformly persistent was proved by Fonda \cite{F}. In precise terms, the system is uniform persistent if and only if there exists $U\subset  X$ an open neighborhood of $0\in U$ and $H:X\to (0,+\infty)$ continuous such that: 
\begin{enumerate}
\item $H(x)=0 \iff x=0$.
\item $\forall\; x\in U\backslash\{0\}\; \exists\, t_x>0:$ $H(\Phi(t_x,x)) > H(x)$. 
\end{enumerate}

This work is motivated by results on persistence which have a significant relevance in biological models, see e.g. 
\cite{BB}, \cite{BBI},  \cite{EPR},  \cite{RW}, \cite{RZ}, \cite{WJW}.
With this in mind, in this paper we study persistence in the 
context of delay differential equations:

\begin{equation}
\label{eq}
x'(t)=f(t,x(t),x(t-\tau))
\end{equation}
where $f:[0,+\infty)\times[0,+\infty)^{2N}\to \R^N$ is continuous and 
$\tau>0$ is the delay.
An initial condition for (\ref{eq}) 
can be 
expressed
in the following way
\begin{equation}
\label{ic}
x_0=\varphi,    
\end{equation}
with $\varphi:[-\tau,0]\to [0,+\infty)^{2N}$ a continuous function and $x_t\in C([-\tau,0],\R^N)$, defined by $x_t(s)=x(t+s)$.

Since the space of initial values is infinite dimensional, the local compactness property does not hold. We shall  consider a \emph{guiding type function} and propose suitable conditions to compensate the lack of compactness in order to extend some results that hold for
the non-delayed case.
\medskip{}{}

Moreover, we shall  study the existence of 
$T$-periodic solutions. With this aim, we shall consider the persistence assumptions of the previous section and obtain appropriate invariant subsets of $X$, which yield  the existence of periodic orbits. 




Finally, we shall prove that, in some sense, the previous conditions are also necessary. Indeed, if they are not fulfilled,  
then accurate extra assumptions imply that zero is a global attractor.



\subsection*{Motivation: Nicholson's Equation}

The Nicholson's \cite{N} blowflies equation reads

\begin{equation}
    \label{nich}
    x'(t)=-d x(t) + p x(t-\tau)e^{-x(t-\tau)},
\end{equation} 
where $x(t)$ represents the population,  $d > 0$ is the mortality rate, $ p> 0$ is the production rate and $\tau>0$ is the delay. For a complete survey on the subject, see  \cite{BBI}. We state some relevant results:  

\begin{enumerate}
    \item If $p\le d$ then there do not exist positive equilibria and $0$ is a global attractor of positive solutions.
    
    \item If $p>d$, there exists only one positive equilibrium, which is asymptotically stable 
    when $p\ge de^2$ for $\tau<\tau^*(p)$ and
    for all $\tau>0$ if $p<de^2$. 
    In addition, if $x$ is a solution with positive initial condition then $\liminf_{t\to+\infty}x(t)\ge \min \{\ln \left(\frac pd\right), 
    e^{-\tau d}\}$. 
    
    \item If $p(t)$ and $d(t)$ are positive $T$-periodic functions then the equation has positive $T$-periodic solutions if $p(t)>d(t)$ for all $t>0$, and otherwise there are no positive $T$-periodic solutions. 
\end{enumerate}

The conclusions can be generalized for the system studied in \cite{AD}:

\begin{equation} 
    \label{nich-sys}
    \left\{\begin{array}{c}
x_1'(t)=-d_1 x_1(t) + b_1x_2(t)+ p_1 x_1(t-\tau)e^{-x_1(t-\tau)}    \\
        x_2'(t)=-d_2 x_2(t) + b_2x_1(t)+ p_2 x_2(t-\tau)e^{-x_2(t-\tau)}  
    \end{array}
    \right.
\end{equation}
\begin{enumerate}
\item If $p_i+b_i\le d_i$ then $0$ is a global attractor of the positive solutions. 

\item There is uniform persistence as long as $p_i+b_i>d_i$.
        
\item There are positive $T$-periodic solutions if $b_i(t), d_i(t), p_i(t)$ are positive $T$-periodic functions and $b_i(t)<d_i(t) <p_i(t)+d_i(t)$ for all $t>0$. 
\end{enumerate}

\medskip

\medskip

\subsection*{Outline}

The paper is organized as follows. In the first section, we set up notation and state a preview of our main results. 
In the second section, we prove some basic facts concerning weak and strong persistence. 
In the third section, we proceed with the study of uniform persistence.
In the fourth section, we analyze the existence of periodic solutions by means of an appropriate fixed point operator. 
Finally, we study the global attractiveness of the trivial equilibrium.

\medskip 

\section{Persistence in delayed systems}

\subsection{Persistence definitions}

In this section, we follow tha approach of  \cite{ST}. 
Let $X$ be a metric space and $\rho(x): X\to [0,+\infty)$ a  continuous function. A semi-flow $\Phi :[0,+\infty)\times X\to  X$ is a continuous function such that

\begin{enumerate}
    \item $\Phi(0,x)=x\quad\,\forall\, x\in X$;
    \item $\Phi(t,\Phi(t,x))=\Phi(t+s,x)\quad\forall\, x\in X, \forall\, s, t\in [0,+\infty).$
\end{enumerate}

We shall say that a semi-flow is

\begin{itemize}  

\item weakly $\rho$-persistent, if
$$\limsup_{t\to+\infty}  
\rho(\Phi(t,x)) > 0\qquad \forall\, x \in  X\backslash\{0\}.$$
    
\item strongly $\rho$-persistent, if
$$\liminf_{t\to+\infty}  
\rho(\Phi(t,x)) > 0\qquad \forall\, x \in X\backslash\{0\}.$$

\item uniformly $\rho$-persistent, if
there exists $\varepsilon > 0$ such that
$$\liminf_{t\to+\infty}  
\rho(\Phi(t,x)) > \varepsilon \qquad \forall\, x \in  X\backslash\{0\}.$$

\end{itemize}

\medskip

\subsection{Preview of the main results}
In order to study the dynamical behaviour of the solutions of a DDE system, we shall establish an appropriate condition to ensure that the induced semi-dynamical system is well-defined; namely, that non-trivial solutions with non-negative initial value  remain positive. Secondly, with the purpose of studying persistence of the associated semi-dynamical system, we shall focus on the case that 0 is a uniform repeller. 
With this aim, we shall consider a suitable guiding function $V$ in order to study the flow in sets of the form $\{V(x)<r\}$.  In this sense, we may consider conditions to guarantee that the flow stays away from the origin. In order to control the delayed part of $f$, we shall assume that $V$ is globally monotone or, 
instead, that $V$ is locally monotone together with a condition on the logarithmic derivative of $V$. Under these specific conditions we will deduce weak, strong and uniform persistence. Indeed, we shall show that all trajectories of the DDE system with positive initial data lie outside a closed ball centered at the origin of $\R^N$ for sufficiently large values of $t$.

In addition, we shall consider extra hypotheses over sets of the form $\{V(x)\ge R\}$, with $R\gg0$, whence the existence of $T$-periodic positive solutions will be proven. 
Finally, we analyze the global attractiveness of the trivial equilibrium by means of an appropriate population-type function. 


\section{Weak and strong persistence}

With population models in mind, 
we shall consider non-negative initial values to study persistence of delay differential equations. 
Additionally, we shall require an extra assumption to guarantee the flow-invariance property. This hypothesis implies that non-trivial solutions with non-negative initial data remain strictly positive.
\medskip{}{}

{\bf (H1)} \label{posit}
If for some $j\in\{1,2,...,N\}$, $x_j=0$  and $y\ne 0$, then $f_j(t,x,y)>0$ for 
all $t>0.$    
\medskip{}{}

\noindent Indeed, suppose that $\min_k x_k(t_0) =x_j(t_0)= 0$ for the first time at some $t_0>0$, then 
$0\ge x_j'(t_0)=f_j(t_0,x(t_0),x(t_0-\tau))>0,$ a contradiction. 
\medskip{}{}

\noindent Hence, the flow associated to (\ref{eq}) is
\begin{equation}
    \label{flow}
    \Phi:[0,+\infty)\times C([-\tau,0],[0,+\infty)^{N})\to C([-\tau,0],[0,+\infty)^{N})
\end{equation}
given by 
$$
\Phi(t,\varphi)=x_t.
$$


\medskip{}{}

Observe that equation 
(\ref{nich})
with $p\le d$ shows that {\bf (H1)} alone does not suffice to 
guarantee the weak persistence of 
solutions with positive initial data.
With this model in mind, we shall consider throughout the paper a $C^1$ Lyapunov-like mapping 
$$V:(0,+\infty)^N\to (0,+\infty)$$  such that 
$$\lim_{\vert x\vert \to 0}V(x)= 0.$$

The obvious example is $V(x):= |x|^2$, where $|\cdot |$ stands for the Euclidean norm of $\R^N$, although many other choices of $V$ could be used in applications. The point is that, in contrast with Lyapunov functions, we shall require that $\dot V>0$ for $x$ close to the origin where, as usual,
$$\dot V(t)=\frac{d{V\circ x}}{dt}.$$

In this sense, $V$ may be compared with the \emph{guiding functions} introduced by Krasnoselskii and successfully extended for DDEs (see e.g. \cite{F2}, \cite{K}, \cite{KRA}) but, unlike the guiding functions, our conditions shall involve 
sets of the form $\{ V(x)<r\}$ instead of $\{ |x|\ge r\}$.

The simplest situation would consist in assuming 
that $0$ is a uniform repeller
for the vector field $f(\cdot,\cdot,y)$ when  $y$  is close to the origin:

\begin{prop}\label{prop1}
Assume {\bf (H1)} holds and 
\medskip 

\noindent {\bf (H2)} there exist $t_0,r_0>0$ such that
$$\langle \nabla V(x),f(t,x,y)\rangle >0 \quad \hbox{for } t>t_0\hbox{ and } V(x),V(y)<r_0.$$
Then the system is weakly persistent.

\end{prop}
  
 \begin{proof}
 Suppose $x(t)$ is a positive solution such that $\lim_{t\to +\infty} 
|x(t)|= 0$ and define $v(t):= V(x(t))$. Because $x(t)\ne 0$ for all $t$, it follows that $v$ is well defined and tends to $0$ as $t\to +\infty$. 

Set $t_n$ such that $v(t_n)=\min_{t\in [0,n]}v(t)$. It is easily seen that $t_n\to +\infty$ and, for $n$ large, 
$$0\ge  v'(t_n)=
\langle \nabla V(x(t_n)),f(t_n,x(t_n),x(t_n-\tau)\rangle, 
$$ 
which is a contradiction.
\end{proof}

\begin{rem}

\begin{enumerate}
    \item Notice, incidentally, {\bf (H2)} alone does not imply necessarily that solutions remain non-negative. 

    \item Observe that when $\tau=0$, {\bf (H2)} implies the hypothesis in \cite{F}.
    
    \item A standard assumption in population models is that $f(t,0,0)=0$, i. e.  that $0$ is an equilibrium point; in this case, {\bf (H2)} obviously implies this equilibrium cannot be asymptotically stable. 

\end{enumerate}
\end{rem}

However, the previous condition is not fulfilled in many models, even in the scalar case. For example, Nicholson's equation (\ref{nich}), which satisfies, instead, a weaker condition when $p>d$, namely $f(x,x)>0$ for $0<x < r:=\ln \frac pd$.  
The situation can be generalized for systems in terms of the vector field $f(\cdot,x,x)$. 
Nevertheless, it is clear that this do not suffice to ensure the persistence of the system because the function $|x(t)|-|x(t-\tau)|$ does not necessarily tend to $0$. We shall avoid this difficulty adding a $V$-monotonicity property. 

\begin{prop}\label{prop2}
Assume {\bf (H1)} holds, 
\medskip 

\noindent {\bf (H3)}\hbox{ there exist }$t_0,r_0>0$\hbox{ such that }
$$\langle \nabla V(x),f(t,x,x)\rangle >0 \hbox{ for } t>t_0,V(x)<r_0$$
and 
\medskip{}{}

\noindent{\bf (H4)}
$\qquad \langle \nabla V(x),f(t,x,y)\rangle \ge \langle \nabla V(x),f(t,x,x)\rangle\;\hbox{ if }\; V(y)\ge V(x).$
\medskip{}{}

\noindent Then the system is strongly persistent.
\end{prop}

\begin{proof}
Suppose exists a sequence $s_n\to +\infty$ such that $x(s_n)\to 0$ and set $t_n$ such that $v(t_n) =\min_{t\in [0,s_n]} v(t)$. For $n$ sufficiently large, it follows that $v(t_n)<r_0$,
$v'(t_n)\le 0$ and $v(t_n-\tau)\ge v(t_n)$.
As before, 
we obtain  
$$0\ge v'(t_n)=\langle \nabla V(x(t_n)),f(t_n,x(t_n),x(t_n-\tau)\rangle
\ge \langle \nabla
V(x(t_n)),f(t_n,x(t_n),x(t_n)\rangle,$$  a contradiction since, by {\bf (H3)}, this latter quantity is strictly positive.

\end{proof}

\begin{rem}
\begin{enumerate}
    \item For example, Theorem 5.2 of Berezansky and Braverman in \cite{BB} with $l=1$ and $h_1(t)=t- \tau$ can be regarded as a consequence of the preceding result with $N=1$,  taking $V(x)=x$.

\item Proposition \ref{prop2} is also verified by Nicholson's system (\ref{nich-sys}), with a non-smooth function $V(x)=\min \{ x_1,x_2\}$ {\rm (see
\cite{AD})}:

$$\nabla V(x_1,x_2) =\left\{\begin{array}{cc}
 (1,0) & x_1 < x_2;  \\
    (0,1)  & x_1 > x_2.
\end{array} 
\right.$$
\end{enumerate}
\end{rem}
\medskip{}{}

Due to the fact that 
the distance between $x(t)$ and $x(t-\tau)$ may  be large, the global monotonicity assumption cannot be replaced by a weaker one such as local monotonicity.
Nonetheless, this condition is verified in (\ref{nich}) with $V(x)=x$, $\eta=1$ and proves to be sufficient in presence of
an extra assumption that controls the logarithmic derivative of $v$:

\begin{prop}
\label{prop3}
Assume that {\bf (H1)} and {\bf (H3)} hold,

\medskip{}

\noindent {\bf (H5) }\hbox{there exists }$\eta>0$\hbox{ such that }$$ \langle \nabla V(x),f(t,x,y)\rangle \ge \langle \nabla V(x),f(t,x,x)\rangle\hbox{ if }V(x)\le V(y)\le \eta\,
$$
and
\medskip{}

\noindent {\bf (H6)}
$\hbox{there exists }k>0\hbox{ such that } \langle \nabla V(x),f(t,x,y)\rangle \ge -k V(x).$

\medskip 

\noindent Then the system is strongly persistent.

\end{prop}

\begin{proof}
Let us begin by noticing that {\bf (H6)} implies 
$$v(t-\tau) \le e^{k\tau}v(t) \hbox{ for all } t\ge \tau.$$
Indeed, since 
$$v'(t)= 
\langle \nabla V(x(t)), f(t,x(t),x(t-\tau)\rangle \ge -k V(x(t))= -kv(t),$$ we deduce that 
$$
\ln v(t) - \ln v(t-\tau) \ge -k\tau,$$
that is: $v(t) \ge e^{-k\tau} v(t-\tau).$
\medskip{}{}

Thus, the proof follows as in the preceding proposition, taking $n$ large enough such  that also
$v(t_n) \le e^{-k\tau}\eta $.

\end{proof}

\medskip{}{}

\begin{rem}
It is an easy exercise now to prove that {\bf (H1)}, {\bf (H2)} and {\bf (H6)}  imply strong persistence. Indeed, weak persistence is obtained under {\bf (H1)}, {\bf (H2)} as stated in Proposition \ref{prop1}. Condition {\bf (H6)}  gives an upper bound for $x(t-\tau)$ in terms of $x(t)$, whence the strong persistence is  achieved. 
\end{rem}

\section{Uniform persistence}

In this section, 
we shall analyze whether or not the  assumptions of the previous section
are sufficient to prove the uniform persistence. 
Specifically, we are interested in obtaining an accurate value of $\mu>0$ such that the set $V^\mu:= V^{-1}(0,\mu)$
is a repeller, 
i. e. all trajectories with positive initial data lie outside the set for sufficiently large values of $t$. 
Uniform persistence follows then from the fact that $V^\mu$ contains a set of the form $\{x\in\mathcal{C}:0<|x|<\xi\}$ for some positive $\xi$.

\medskip
\begin{rem}
It is worthy mentioning that uniform persistence is easily deduced under the assumptions {\bf(H1)}, {\bf(H2)} combined with {\bf(H4)} or {\bf(H6)}. 

\end{rem}{}

In view of the preview remarks, we shall restrict our attention to analyze the weaker conditions of Propositions \ref{prop2} and \ref{prop3}.

\medskip

With this in mind, recall 
{\bf(H3)}, that is
$$\langle \nabla V(x),f(t,x,x)\rangle >0 \qquad \hbox{for } t>t_0,\; V(x)<r_0$$
and let us firstly observe if we choose $i<r_0$ such that
$$
\liminf_{t\to+\infty} v(t) = i,
$$ 
then three different situations may be considered: 

\begin{enumerate}
    \item \label{caso1}
    $v(t)\ge i$ for all $t\gg 0$. Then we may choose as before a sequence $(t_n)_{n\in\N}$ such that 
     $$
     \lim_{n\to+\infty}v(t_n)=i,
     $$ 
     $v'(t_n)\le 0$ and $v(t_n-\tau) \ge v(t_n)$. 
     Thus, a contradiction yields under {\bf (H4)}.
     Alternatively, if we  assume {\bf(H5)} and {\bf(H6)}, then we obtain a contradiction provided that $i< e^{-k\tau}\eta$. 
    
    \item \label{caso2}
    $v(t)$ oscillates around $i$. Then we may choose 
    a sequence $t_n\to +\infty$ such that $v(t_n)\to i^-$  and $v'(t_n)\le 0$. For instance, we may set $s_n\to +\infty$ such that $v(s_n)>i$ and then $(t_n)_{n\in\mathbb{N}}$ satisfying $v(t_n)=\min_{[s_n,+\infty)} v(t)$. Notice, however, that if $t_n$ is close to $s_n$, then it may happen that $v(t_n-\tau)< v(t_n)$ and 
    {\bf (H3)} and {\bf (H4)} are not of any help. Observe, for future considerations, that in this latter situation $v(t_n-\tau)\to i^-$ as well.  
    
    
    \item\label{caso3}  
     $v(t)\to i^-$ as $t\to+\infty$. Again, the previous conditions cannot be applied in a direct way. Here, it may even occur that no sequence $(t_n)_{n\in\N}$ with $v'(t_n)\le 0$ exists. 
        
        
\end{enumerate}

In order to take a closer look to the last two  `tricky cases',
let us recall that, from the previous section, under 
assumptions {\bf(H1)}, {\bf(H3)} together with {\bf(H4)} or {\bf(H5)}--{\bf(H6)}
we already know that $i>0$. 
\medskip 

To fix ideas, consider firstly the autonomous scalar equation, i. e. $N=1$ and $f$ does not depend on $t$.
Suppose that {\bf(H3)} is verified with $V(x)=x$; then, for each $i\in (0,r_0)$ there exist $\delta=\delta(i)>0$ and $c=c(i)>0$ such that 
$f(x,y)>c$ for $x,y\in(i-\delta,i)$. This, in turn, implies:

\begin{enumerate}
\item If $x(t_n-\tau), x(t_n)\to i^-$ and $x'(t_n)\le 0$ then 
$0\ge f(x(t_n),x(t_n-\tau))$, a contradiction because  $f(x(t_n),x(t_n-\tau))>0$ when $n$ is large. 
    \item If $x(t)\to i$ as $t\to +\infty$, then $x'(t)> c$ for $t\gg 0$ whence $x(t)\to +\infty$, a contradiction. 
    
\end{enumerate} 

\medskip{}

This can be generalized for a system, as shall be described in the  following result.

\begin{thm}
\label{up1}
Assume that {\bf (H1)}, {\bf (H4)} hold and 
\medskip

\noindent {\bf (H7)}\hbox{ there exists }$r_0>0$  
$$\hbox{such that for all }
i\in (0,r_0) \liminf_{t\to+\infty, V(x),V(y)\to i^-}\langle \nabla V(x),f(t,x,y)\rangle>0.
$$

\noindent Then the system is uniformly persistent.
\medskip

More precisely, all solutions of (\ref{eq})-(\ref{ic}) 
with $\varphi_j(t)\geq 0$ for all $j$ and 
\hbox{$t\in [-\tau,0]$}
satisfy 
$$\liminf_{t\to+\infty} V(x(t))\ge r_0,$$
that is, $x(t)\notin V^{r_0}$ for $t$ sufficiently large.
\end{thm} 
\medskip{}{}

A slightly different conclusion holds when the global monotonicity assumption {\bf (H4)} is replaced by {\bf (H5)} and {\bf (H6)}. 

\begin{thm}
\label{up2} Assume that {\bf (H1)}, {\bf (H5)}, {\bf (H6)} and {\bf (H7)} hold,
then the system is uniformly persistent.

\medskip{}{}

More precisely, all solutions of (\ref{eq})-(\ref{ic}) 
with $\varphi_j(t)\geq0$ for all $j$ and $t\in [-\tau,0]$ 
satisfy 
$$\liminf_{t\to+\infty} V(x(t))\ge \min\{ r_0, e^{-k\tau}\eta \}.$$

\end{thm} 

\subsection{Proofs of Theorem \ref{up1} and Theorem \ref{up2} }

\begin{proof}
Suppose $\liminf_{t\to+\infty} v(t)=i \in (0,r_0)$
and fix a positive constant $c>0$  such that
$$ 
\liminf_{t\to+\infty, V(x),V(y)\to i^-}
\langle \nabla V(x),f(t,x,y)\rangle>c.
$$ 
If  $v(t)\to i^-$,
then for all $t\gg 0$  
$$
v'(t)=\langle \nabla V(x(t)),f(t,x(t),x(t-\tau)) >c
$$
a contradiction. Thus, we deduce there exists a sequence $s_n\to +\infty$ such that $v(s_n)>i$ and $v(s_n)\to i^+$.  
Take as before $t_n \in [s_1,s_n]$ such that $v(t_n)= \min_{t\in [s_1,s_n]} v(t) $, then $v(t_n)\to i$ and, 
for $n$ large,  $v'(t_n)\le 0$ and $v(t_n)\le v(t_n-\tau)$. This yields a contradiction when {\bf (H4)} holds or when 
{\bf (H5)}--{\bf (H6)} hold, provided that also $i< e^{-k\tau}\eta.$


\end{proof}{}

\medskip{}{}{}

Lastly, inspired by Nicholson's equation, it is worthy mentioning that {\bf (H7)} is not satisfied in many situations when $N>1$: consider, for instance $V(x)=\vert x\vert^{2}$ and
$f(x,y)=y$. 
However, as already observed, most models assume that 
$f(t,0,0)=0$ and, consequently, it is expected that, as the solution gets close to $0$, the 
distance between $x(t)$ and $x(t-\tau)$ is comparatively small. This means that the condition might be relaxed by assuming that the inequality holds only for $|x-y|$ small enough. In more precise terms, we may define 
$$\theta(i):= \limsup_{t\to+\infty, V(x), V(y)\to i} |f(t,x,y)|
$$
and observe that if 
$v(t)\to i^-$ 
then from the mean value inequality 
$$|x(t)-x(t-\tau)|\le \tau |x'(\xi)|\qquad \xi \in (t-\tau,t)$$
we deduce: 
$$
\limsup_{t\to+\infty} |x(t)-x(t-\tau)| \le \tau \theta(i).
$$
Thus, we obtain:
\medskip{}

\begin{thm}
Theorem \ref{up1} and Theorem \ref{up2} are still valid if {\bf (H7)} is replaced by: 
$${\bf (H8)}\hbox{ There exists } r_0>0 \hbox{ such that, for all } i\in (0,r_0) \hbox{ and some } C(i)>\tau \theta(i)$$ 
$$\liminf_{t\to+\infty, V(x),V(y)\to i^-, |x-y|\le C(i) }\langle \nabla V(x),f(t,x,y)\rangle  > 0.
$$ 
\end{thm}

\section{Periodic solutions}

When searching for periodic orbits of an ODE system, it is usual to employ a solution  operator such as the Poincar\'e map and apply a standard procedure using the Brouwer degree to obtain fixed points. For the delayed case, since the space of initial values is infinite dimensional, the Brouwer degree cannot 
be applied: we shall use instead Leray-Schauder degree techniques.
\medskip{}{}

Let us recall that the Leray-Schauder degree is defined in this context as follows.
Let $C_T$ be the Banach space of continuous $T$-periodic vector functions, equipped with the standard norm $\|\cdot \|_\infty$. 
Let $U\subseteq C_T$ be open and bounded, and let $K:\overline U\to C_T$ be a compact operator such that $K x\neq x$ for $x\in\partial U$. 
Set $\varepsilon=\inf_{x\in\partial U} \| x-Kx\|_\infty$, and define
$$
{\rm deg}_{L-S}(I - K,U,0) =  {\rm deg}_{B}(\left.(I-K_{\varepsilon})\right|_{V_{\varepsilon}},U\cap V_{\varepsilon} ,0),
$$
where $K_{\varepsilon}$ is an $\varepsilon$-approximation of $K$ with ${\rm Im} (K_{\varepsilon})\subseteq V_{\varepsilon}$ and ${\rm dim}(V_{\varepsilon}) < \infty$.
\medskip

We will show that the Leray-Schauder degree of the operator $I-K$ is non-zero on an appropriate subset of the positive cone $\mathcal{C}\subset C_T$ and therefore the set of fixed points of the compact operator $K$ is non-empty.
\medskip

Inspired by population models, in order to state our result we shall impose a condition that allows to find an upper bound for the flow.
With this aim, we may choose a continuous function $a:[0,+\infty)\to (0,+\infty)$ and define: 

\begin{equation}
\begin{array}{lcl}
F(t,x,y)&=& 
\langle \nabla V(x),f(t,x,y)\rangle + a(t)V(x),\\
F^{*}(t,r)&=& \sup_{V(x),V(y)\le r} {\frac {F(t,x,y)}{a(t)}}.
\end{array}
\end{equation}
For example in Nicholson's equation with $V(x)=x$, we may choose $a(t)=d$, where $d$ is the mortality rate.
\medskip

We can now formulate the main result of this section:  
\begin{thm}
\label{ps1}
Assume that  $f$ is $T$-periodic in the first coordinate,  {\bf (H1)}, {\bf (H3)}, {\bf (H4)} hold and $$\hbox{{\bf (H9)} } F^{*}(t,R)<R \hbox{ for $0\le t\le T$ and some 
$R>r_0$ } $$
where $r_0$ is the constant in {\bf (H3)}.
Then there exists at least one $T$-periodic positive solution of (\ref{eq})-(\ref{ic}) in $\Omega=\{x\in [0,+\infty)^N:V(x)\in(r_0,R)\}$ provided that the Euler characteristic of $\Omega$ is non-zero. 
\end{thm}

Alternatively, we have:

\begin{thm}
\label{ps2}
If  $f$ is $T$-periodic in the first coordinate, {\bf (H1)}, {\bf (H3)}, {\bf (H5)}, {\bf (H6)} and {\bf (H9)} hold, then there exists at least one $T$-periodic positive solution of (\ref{eq})-(\ref{ic}). 
\end{thm}

The proof of our main theorem shall be based on the following crucial result (see \cite{hopf}):

\begin{thm}[Hopf Theorem]
If $\nu$ is the outward normal on an
oriented, compact manifold $M$, then the degree of $\nu$ equals the Euler characteristic of $M$.
\end{thm}

\subsection{Proof of Theorem \ref{ps1} and Theorem \ref{ps2}}

\begin{proof}
Let us begin by introducing the standard continuation method \cite{maw}, following the notation of  \cite{AKR}. 
For a function $x\in C_T$, let us
write
$$\mathcal Ix(t):= 
\int_0^t x(s)\, ds, \qquad 
\overline x:= \frac{1}{T}\,{\mathcal I x(T)}.
$$

Moreover, denote by $\mathcal N$ the Nemitskii operator associated to the problem, namely 
$$
\mathcal Nx(t):= f(t,x(t),x(t-\tau)). 
$$

Let us consider the open bounded sets $\Omega\subseteq\R^N$, $U=\{x\in \mathcal{C}:x(t)\in \Omega\,\hbox{ for all $t>0$} \}\subseteq C_T$ and define the compact operator 
$K:\mathcal{C}\to C_T$ by
$$
Kx(t):= \overline x -t\, \overline {\mathcal Nx} + 
\mathcal I\mathcal Nx(t) - \overline{\mathcal I\mathcal Nx}.
$$
First, let us observe that $K$ is well-defined, i.e. $Kx$ is $T$-periodic.
Fix $x\in \mathcal{C}$, since $f$ is $T$-periodic,
$$
Kx(t+T)-Kx(t)=-T\,\overline{\mathcal Nx}+
\int_{t}^{t+T} \mathcal Nx(s)\, ds= -{\mathcal Nx}+
\int_{0}^{T} \mathcal Nx(s)\, ds=0.
$$

Moreover, via the Lyapunov-Schmidt reduction, if $x\in \mathcal{C}$ is a fixed point of $K$ then $x$ is a solution of the equation. Indeed, a fixed point of $K$ verifies $\overline{\mathcal Nx}=0$ whence $x'=\mathcal Nx$.

Let $K_0:\overline{U} \to \R^N$ be the continuous function $K_0x:= \overline x -\frac{T}{2}\,
\overline {\mathcal Nx}$  and consider, for $\lambda\in [0,1]$, the homotopy $K_{\lambda}:=\lambda K +(1-\lambda)K_0$. 
As before, for $\lambda>0$ we know that $x\in\overline U$ is a fixed point of $K_\lambda$ if and only if
$x'(t)=\lambda\mathcal Nx(t)$, that is:
$$
x'(t)= \lambda f(t, x(t),x(t-\tau)).
$$
We claim that $K_\lambda$ has no fixed points on $\partial U$.
This follows from the fact that the flow is inwardly pointing on $\partial\Omega$. Indeed, if
$x\in\overline\Omega\cap \{x_i = 0\}$, the result is deduced from {\bf (H1)}. If $x\in\{V=R\}$ or $x\in\{V=r_0\}$ then the outward normal is $\nabla V(x)$ and $-\nabla V(x)$, respectively. Then notice that {\bf (H9)} implies that if $V(x)=R$ then 
$\langle \nabla V(x),f(t,x,y)\rangle <0$ and on the other hand, {\bf (H2)} implies $\langle \nabla V(x),f(t,x,y)\rangle>0$ when  $V(x)=r_0$.

Let us identify 
 $\R^N$ with the set of constant functions of $C_T$,  whence $U \cap \,\R^N=\Omega$. Since the range of $K_0$ is contained in $\R^N$ the Leray-Schauder degree of $I-K_0$ can be computed as the Brouwer degree of its restriction to  $\Omega$. 

Apply another homotopy,  
$$
H(x,\lambda)=\lambda K_0(x)-(1-\lambda)\nu(x),\qquad\hbox{for $(x,\lambda)\in\overline{\Omega}\times[0,1]$},$$
where $\nu$ is the outward normal, which by {\bf (H9)} does not have fixed points on $\partial\Omega$.
By the homotopy invariance of the degree and Hopf theorem, we conclude that
$$
deg_{LS}(I-K,U,0)=deg_{B}(I-K_0,\Omega,0)=deg_{B}(-\nu,\Omega,0)=(-1)^N\chi(\Omega)\neq 0.$$

\end{proof}

\subsection{Global attractiveness of the 
trivial equilibrium}

In this section, we analyze the global attractiveness of the 
trivial equilibrium. 
We shall prove that, in some sense, the previous conditions are also necessary: more precisely, we shall see that if they are not fulfilled,  
then accurate extra assumptions imply that zero is a global attractor.   

Let us define
$$
\phi^*(r):= 
\sup_{t\ge 0} {F^*(t,r)},
$$
and assume that $\phi^*$ is continuous.

\medskip 

\begin{thm}
Assume that for every $\varepsilon>0$
there exists $\mu>0$ such that $\{x: V(x)\in(0,\mu)\} \subset B_\varepsilon(0)$  and suppose  there exists $R_0$ such that  $\phi^*(r)<r$ for $0<r\le R_0$. 
Then every solution with initial data $\varphi$, such that $\varphi_{j}(t)< R_0$ for all $j$ and $t\in [-\tau,0]$, tends to $0$ as $t\to +\infty$.  

\end{thm}

\begin{proof}
Observe, in the first place, that if $v\le r$ on $[t-\tau, t]$ and
$v'(t)\ge 0$, then 
$$a(t)v(t)\le F(t,x(t),x(t-\tau))
$$
that is,
$$
v(t) \le F^{*}(t,r). 
$$
This implies either that $v(t)$ is initially decreasing or else $v(t)\le R_1:={\phi^*(R_0)}<R_0$ for all $t\ge 0$. Furthermore, if $v(t)$ decreases and enters into the interval $(0,R_1]$ at some value $t_1$, then it remains there for all $t\ge t_1$.  Repeating the reasoning for $R_{k+1}:= {\phi^*(R_k)}<R_k$, two different situations may occur:

\begin{enumerate}
    \item There exists $t_k\to +\infty$ such that $v(t)\in (0,R_k]$ for $t\ge t_k$. 
    
\item There exist $k$ and $t_k$ such that $v(t) \in (R_{k+1},R_k]$ and decreases strictly for $t\ge t_k$.

\end{enumerate}

Since the sequence $(R_k)_{k\in\N}$ is strictly decreasing, it follows that 
$$
\phi^*(\lim_{j\to \infty} R_k) = \lim_{j\to \infty} \phi^*(R_k)= \lim_{j\to \infty} R_k. 
$$
Due to the fact that $\phi^*(r)<r$, we deduce that $R_k\to 0$ and hence, in the first case, that $v(t)\to 0$ as $t\to \infty$. 

In order to complete the proof, we shall prove that the second situation cannot happen. Indeed, if $v(t)\to r>0$ as $t\to+\infty$, then we may fix $\tilde r>r$ such that $\phi^*(\tilde r)<r$ and $\tilde t$ such that $v(t)\le \tilde r$ for $t\ge \tilde t$. It follows that 
$v(t)\le \phi^*(\tilde r)<r$ for $t\ge \tilde t+\tau$, which contradicts the fact that $v(t)\to r$.
Hence, for $\varepsilon >0$ we may fix $\mu >0$ 
such that $|x|<\varepsilon$ for $V(x)<\mu$. Set $t_0$ such that $v(t)<\mu$ for $t>t_0$ 
and it follows that $|x(t)|<\varepsilon$ for $t>t_0$. 

\end{proof}

\begin{exa}

In Nicholson's model (\ref{nich}) with $p\le d$ and $V(x)=\frac {x^2}2$ we have that $a=2d$,  $F(x,y)=pxye^{-y}$ and
$$\phi^*(r)= \max_{x,y\le \sqrt{2r}} \frac p{2d} xye^{-y}
\le
\left\{\begin{array}{cc}
re^{-\sqrt{2r}}     &  \hbox{if }\, r\le 1\\
\frac {\sqrt{2r}}{2e}     & \hbox{if }\, r > 1.
\end{array}\right.
$$ 
Thus, $\phi^*(r)<r$ for all $r>0$, so we conclude that $0$ is a global attractor for all positive solutions. 
The result is still true if $p$ and $d$ are positive continuous functions, provided that $p(t)\le d(t)$
for all $t$.

\end{exa}

\end{document}